\documentclass{amsart}

\usepackage{amsfonts,amssymb,amsmath,amsthm}
\usepackage{tikz}
\usepackage{caption}
\usepackage{subcaption}

\newcommand{\kk}{\mathbb{K}}

\newcommand{\B}[2]{B^{#2}_{\nwarrow#1}}
\newcommand{\RB}[2]{{RB_{#1\nearrow}}^{#2}}

\DeclareMathOperator{\borel}{Borel}
\DeclareMathOperator{\Borel}{Borel}
\DeclareMathOperator{\gens}{gens}
\DeclareMathOperator{\Span}{span}
\DeclareMathOperator{\lex}{lex}

\newtheorem{thm}{Theorem}[section]
\newtheorem{cor}[thm]{Corollary}
\newtheorem{lem}[thm]{Lemma}
\newtheorem{prop}[thm]{Proposition}

\newtheorem{question}[thm]{Question}
\newtheorem{proposition}[thm]{Proposition}
\newtheorem{corollary}[thm]{Corollary}

\newtheorem{lemma}[thm]{Lemma}

\theoremstyle{definition}
\newtheorem{definition}[thm]{Definition}
\newtheorem{defn}[thm]{Definition}

\newtheorem{example}[thm]{Example}
\newtheorem{notation}[thm]{Notation}
\newtheorem{remark}[thm]{Remark}

\usepackage{hyperref}

\begin{document}

\title{The Rees algebra of a two-Borel ideal is Koszul}

\author[M.~DiPasquale]{Michael DiPasquale}     
\address{Michael DiPasquale\\     
	Department of Mathematics\\     
	Oklahoma State University\\     
	Stillwater\\
	OK \ 74078-1058\\     
	USA}     
\email{mdipasq@okstate.edu}
\urladdr{\url{http://math.okstate.edu/people/mdipasq/}}   
\author[C.A.~Francisco]{Christopher A. Francisco}
\address{Christopher A. Francisco\\
	Department of Mathematics\\     
	Oklahoma State University\\     
	Stillwater\\
	OK \ 74078-1058\\     
	USA}    
\email{chris.francisco@okstate.edu}
\urladdr{\url{https://math.okstate.edu/people/chris/}}
\author[J.~Mermin]{Jeffrey Mermin}
\address{Jeffrey Mermin\\     
	Department of Mathematics\\     
	Oklahoma State University\\     
	Stillwater\\
	OK \ 74078-1058\\     
	USA}
\email{mermin@math.okstate.edu}     
\urladdr{\url{https://math.okstate.edu/people/mermin/}}   
\author[J.~Schweig]{Jay Schweig}
\address{Jay Schweig\\
	Department of Mathematics\\     
	Oklahoma State University\\     
	Stillwater\\
	OK \ 74078-1058\\     
	USA}
\email{jay.schweig@okstate.edu}
\urladdr{\url{https://math.okstate.edu/people/jayjs/}}

\author[G.~Sosa]{Gabriel Sosa}
\address{Gabriel Sosa\\
	Department of Mathematics and Statistics\\     
	Amhurst College\\     
	Amherst\\
	MA \ 01002\\     
	USA}
\email{gsosa@amherst.edu}
\urladdr{\url{https://www.amherst.edu/people/facstaff/gsosa/}}

\begin{abstract}
Let $M$ and $N$ be two monomials of the same degree, and let $I$ be the smallest Borel ideal containing $M$ and $N$. We show that the toric ring of $I$ is Koszul by constructing a quadratic Gr\"obner basis for the associated toric ideal. Our proofs use the construction of graphs corresponding to fibers of the toric map. As a consequence, we conclude that the Rees algebra is also Koszul.
\end{abstract}

\maketitle

\section{Introduction}\label{s:intro}


An arbitrary graded ring $R$ over a field $R_0=\kk$ is \textit{Koszul} if the residue field $R/R_+\cong \kk$ has a linear resolution over $R$.  If $R\cong S/J$ is the quotient of a polynomial ring $S$ by an ideal $J\subset S$,  then $R$ is Koszul if $J$ has a Gr\"obner basis consisting of quadrics with respect to some monomial order \cite[Section 6.1]{EH}.  Our focus in this paper is when Rees algebras associated to certain Borel ideals are Koszul, and we shall approach this question by determining when the defining ideal of the Rees algebra $R(I)$ has a Gr\"obner basis consisting of quadrics.

Conca and De Negri show that the Rees algebra of a principal Borel ideal (the smallest Borel ideal containing a given monomial) is Koszul, Cohen-Macaulay, and normal. (See, e.g., \cite{BrunsConca,DN99}.)  However, they produce examples of ideals with three Borel generators that are none of the above:

\begin{example}(\cite[Example 1.3]{BrunsConca})\label{e:badthreeborel}
 Consider the smallest Borel ideal containing the monomials $a^{3}c^{3}$, $b^{6}$, and $a^{2}b^{2}c^{2}$.  Then the cubic syzygy $(a^{3}c^{3})^{2}(b^{6})=(a^{2}b^{2}c^{2})^{3}$ is minimal.  In particular, the Rees algebra has a minimal generator in degree three. Moreover, the Rees algebra is neither normal nor Cohen-Macaulay.
\end{example}

These examples naturally raise the question of how the Rees algebras of Borel ideals with two Borel generators behave, which Conca posed at the conference honoring Craig Huneke in July 2016.

\begin{question}[Conca]\label{q:thequestion}
Let $I$ be a Borel ideal with two Borel generators.  Is the Rees algebra of $I$ necessarily Koszul?
\end{question}

In fact, the Rees ideal of an ideal with three Borel generators can have generators of arbitrarily high degree, as the following generalization of de Negri's example shows.  Thus the interest is appropriately concentrated on two-generated Borel ideals.

\begin{example}\label{e:threereallybad}
Let $I$ be the smallest Borel ideal containing $f=a^{r}c^{r(r-2)}$, $g=b^{r(r-1)}$, and $h=a^{r-1}b^{r-1}c^{(r-1)(r-2)}$.  Then the syzygy $f^{r-1}g=h^{r}$ is minimal and represents a degree $r$ generator for the Rees algebra of $I$.
For small values of $r$ ($r\le 10$), computations in Macaulay2~\cite{M2} show that the Rees algebra for $I$ is likewise not Cohen-Macaulay or normal.
\end{example}

\begin{remark}\label{r:concacontd}
Unlike in the principal Borel case, the Rees algebra of a two-Borel-generated ideal is almost never normal. For example, the Rees algebra of the smallest Borel ideal containing $a^2c^2$ and $b^4$ is not normal. Meanwhile, limited computational evidence suggests that Rees algebras of two-Borel-generated ideals are Cohen-Macaulay.
\end{remark}

The present paper gives a positive answer to Conca's question \ref{q:thequestion} for equigenerated Borel ideals with two generators.

\begin{remark}\label{r:sorting}
In the case of a principal Borel ideal, De Negri~\cite{DN99} provides an explicit Gr\"obner basis of quadrics using the operation Sturmfels calls \textit{sorting}~\cite{S96}.  Similar methods can be used to give a Gr\"obner basis of quadrics whenever the ideal is \textit{closed under sorting}.  Two-Borel-generated ideals are almost never closed under sorting, so we must use different techniques to show the existence of a Gr\"obner basis of quadrics in this case.
\end{remark}


The paper is structured as follows.  In Section \ref{s:borelnotation} we define notation used throughout the paper and recall the important definitions for Borel ideals and generators.  In Section \ref{s:reesnotation} we recall important definitions about the Rees algebra, and we define our notation for the variables in the Rees (and toric) ring.  Section \ref{s:graphs} defines a graph associated to any multidegree in a Borel ideal, which will be the key ingredient in our proof.  Section \ref{s:toricproof} contains the proof that the toric ideal of a two-generated Borel ideal is Koszul.  Finally, Section \ref{s:reesproof} translates this result to the Rees ideal.


\section{Notation and background for Borel ideals}\label{s:borelnotation}
Let $R=\kk[x_{1},\dots, x_{n}]$ be the polynomial ring in $n$ variables over an arbitrary field $\kk$.  (In the examples, we refer to the variables as $a,b,c$ instead of $x_{1},x_{2},x_{3}$.)

\begin{definition}\label{d:equigenerated}
A monomial ideal $J\subset R$ is \emph{equigenerated in degree $d$} if its minimal monomial generators all have degree $d$, and simply \emph{equigenerated} if it is equigenerated in some degree $d$.
\end{definition}

Throughout the paper, all ideals of $R$ will be equigenerated unless otherwise stated.

We now recall standard definitions for Borel ideals and Borel generators.  Experts may safely jump to Notation \ref{n:MN}, where we introduce some paper-specific notation, or to Definition \ref{d:sigma}, which introduces the cumulative exponent vector.

\begin{definition}\label{d:borelmove}
Fix a monomial $m\in R$, and suppose that $x_{j}$ divides $m$.  Then for any $i<j$, the operation $\B{j}{i}(m)$, which replaces $m$ with $\frac{x_{i}}{x_{j}}m$, is called \emph{the Borel move replacing $x_{j}$ with $x_{i}$} or simply a \emph{Borel move}.
\end{definition}

\begin{definition}\label{d:borel}
A monomial ideal $I\subset R$ is called \emph{Borel} if it is closed under Borel moves.  More explicitly, $I$ is Borel if, whenever $x_{j}f\in I$ and $i<j$, we must have $x_{i}f\in I$ as well.
\end{definition}

Borel ideals are important because they occur as generic initial ideals (see \cite{BS,Ga}), and they have been studied extensively because the combinatorial condition in their definition makes them susceptible to combinatorial techniques. (See \cite[Section 28]{Pe} for some of this flavor.)

We also define a reverse Borel move as follows.

\begin{definition}\label{d:reverseborel}
Fix a monomial $m\in R$, and suppose that $x_{j}$ divides $m$.  Then for any $k>j$, the operation $\RB{j}{k}(m)$, which replaces $m$ with $\frac{x_{k}}{x_{j}}m$, is called \emph{the reverse Borel move replacing $x_{j}$ with $x_{k}$} or simply a \emph{reverse Borel move}.
\end{definition}

\begin{remark}
 Reverse Borel moves are generally not studied because they are simply Borel moves with a different order on the variables.  So the definition is worth making only if we are already studying regular Borel moves on the same monomials.  (The only ideals closed under both Borel moves and reverse Borel moves are powers of the maximal ideal.  Nevertheless, we will need to study both kinds of moves simultaneously in Section \ref{s:graphs}.)
\end{remark}

\begin{definition}\label{d:borelorder}
  Fix a degree $d$.  We define a partial order $<$, called the \emph{Borel order}, on the degree $d$ monomials of $R$ by setting $m<m'$ whenever $m$ can be obtained from $m'$ by a sequence of Borel moves.  In this case, we say that $m$ \emph{precedes} $m'$ in the Borel order.
\end{definition}

\begin{remark}
  The Borel order, $<$, is not a total order and consequently not a term order.  On the other hand, the usual term orders ``lex'' and ``revlex'' are refinements of $>$ and not of $<$.  (Unfortunately, we cannot simultaneously respect the subscripts by making $x_{1}<x_{2}$ and respect the standard orders by making $x_{1}>x_{2}$.  We have found that respecting the subscripts leads to somewhat less confusion.)  To minimize this notational confusion, we generally use English phrases like ``$m$ comes before $m'$ in the Borel order'' rather than purely symbolic statements like ``$m<m'$''.  
\end{remark}

\begin{definition}\label{d:bgens}
  Let $\mathcal{B}$ be a set of degree $d$ monomials.  Then the smallest Borel ideal containing $\mathcal{B}$ is called the \emph{Borel ideal generated by $\mathcal{B}$}, written $I=\Borel(\mathcal{B})$.  In this case, $\mathcal{B}$ is called a \emph{Borel generating set} for $I$.  $I$ has a unique minimal Borel generating set (namely, its latest monomial generators in the Borel order), whose elements are called its \emph{Borel generators}.  If $I=\Borel(m)$ has only one Borel generator, we say that it is a \emph{principal Borel ideal}.
\end{definition}

For more on Borel generators and principal Borel ideals, see \cite{FMS}.

\begin{definition}\label{d:twoborel}
  Suppose $I=\Borel(M,N)$ has exactly two Borel generators.  We say that $I$ is a \emph{two-Borel} ideal.
\end{definition}

\begin{proposition}\label{p:twoborelissum}
  Let $I=\Borel(M,N)$ be the two-Borel ideal generated by $M$ and $N$.  Then $I$ is the sum of the principal Borel ideals generated by $M$ and $N$, $I=\Borel(M)+\Borel(N)$.  
\end{proposition}

\begin{notation}\label{n:MN}
  Suppose $I$ is a two-Borel ideal.  By convention, we set $M$ equal to the lex-earlier of the two Borel generators, and $N$ equal to the lex-later generator.  We refer to monomials of $\Borel(M)$ as $m_{i}$ and monomials of $\Borel(N)\smallsetminus \Borel(M)$ as $n_{i}$. 
\end{notation}

\begin{definition}\label{d:sigma}
  For a monomial $m\in R$, write $m = x_1^{a_1}x_2^{a_2} \cdots x_n^{a_n}$.
  Then the $n$-tuple $(a_{1},\dots, a_{n})$ is called the \emph{exponent vector} of $m$.  We define a new vector, which we call the \emph{cumulative exponent vector} $\sigma(m) = (\sigma_1(m), \sigma_2(m), \ldots, \sigma_n(m))$, by the rule 
\[
\sigma_i(m) = a_i + a_{i+1} + \cdots + a_{n}.
\]
\end{definition}
\begin{example}\label{e:sigma}
  Suppose $m = a^2 cd^2e\in \kk[a,b,c,d,e]$.  Then $\sigma(m) = (6, 4, 4, 3, 1)$.  Also, $\sigma_1(m) = 6, \sigma_2(m) = \sigma_3(m) =4$, etc.
\end{example}

  A few properties of the cumulative exponent vector $\sigma$ are immediate:
  \begin{proposition}\label{p:basicsigma}
    Let $m$ be a monomial.  Then:
  \begin{enumerate}
  \item $\sigma_1(m)$ is the degree of $m$.  
  \item $\sigma_1(m) \geq \sigma_2(m) \geq \cdots \geq \sigma_{n}(m)$. 
  \item $\sigma_i(m) \neq \sigma_{i+1}(m)$ if and only if $x_i \text{ divides } m$.  
  \end{enumerate} 
\end{proposition}

  \begin{lemma}\label{l:sigmaformula}  
    Suppose $x_{i}$ divides $m$ and $j>i$. (I.e., $\RB{i}{j}(m)$ exists.)  Then
    \[
    \sigma_k\left(\frac{x_j}{x_i}m\right) = \left\{
    \begin{array}{ll}
      \sigma_k(m) & k \leq i \\
      \sigma_k(m) + 1 & i < k \leq j \\
      \sigma_k(m) & j < k
    \end{array}
    \right.
    \]
  \end{lemma}

\begin{proposition}\label{p:borelsigma}
  If $m$ and $m'$ are two monomials of equal degree, then $m \in \borel(m')$ if and only if $\sigma_i(m) \leq \sigma_i(m')$ for all $i$.
\end{proposition}
  
\begin{lem}\label{lem:borel1}
  Suppose $m\in\Borel(m')$ and $\sigma_j(m) \neq \sigma_j(m')$.  Then there exists an index $i < j$ such that the reverse Borel move $\RB{i}{j}(m)$ is contained in $\Borel(m')$.
\end{lem}

\begin{proof}
  Let $i$ be the greatest index  less than $j$ such that $\sigma_i(m) \neq \sigma_j(m)$.  Then we must have $\sigma_i(m) > \sigma_{i+1}(m) = \sigma_{i+2}(m) = \cdots = \sigma_j(m)$.  Now, whenever $i<k\leq j$,  we have $\sigma_{k}(m) = \sigma_{j}(m) <\sigma_{j}(m')\leq \sigma_{k}(m')$, so in particular $\sigma_{k}(m)<\sigma_{k}(m')$.  
Applying Lemma \ref{l:sigmaformula}, we have $\sigma_k(\frac{x_j}{x_i}m) \leq \sigma_k(m')$ for all such $k$, proving the lemma.
\end{proof}

\section{Notation and background for toric and Rees ideals}\label{s:reesnotation}

Let $I \subseteq R$ be an equigenerated monomial ideal with minimal generating set $\gens(I) = \{w_1, w_2, \ldots, w_t\}$.  We associate two pairs of rings and ideals to $I$, namely its toric ideal and Rees ideal.  The toric ideal is our main object of study.

\begin{definition}\label{d:toric}
Let $S=S_I$ denote the polynomial ring $\kk[Y_w : w\in \gens(I)]$, with a variable $Y_w$ for each generator $w$ of $I$.  The \emph{toric map} is the map $\phi: S \to R$ given by 
\[
\phi(Y_{w})=w,
\]
and extended algebraically.  (Here and throughout, we allow for a free re-indexing of the $w_i$ to avoid unwieldy double subscripts.) The \emph{toric ideal} of $I$, which we write $T(I)$, is the kernel of $\phi$.  The \emph{toric ring} of $I$, denoted $\kk[I]$, is the quotient $S_I/T(I)$ (naturally isomorphic to the image of $\phi$, which is a subring of $R$).
\end{definition}

\begin{notation}\label{n:multidegree}
The toric ring $S_I$ 
inherits the multigrading from $R$.  That is, we have $S_{I}=\displaystyle\bigoplus_{\mu} S_{\mu}$, where $\mu$ ranges over the monomials of $R$, and $S_{\mu}$ is the $\kk$-vector space
\[
S_\mu:=\Span_{\kk}\{Y=\prod Y_w^{a_w}\in S: \phi(Y)=\mu \}.
\]
Observe that,
if the generating degree of $I$ is $d$, then $S_{\mu}=0$ whenever $\mu$ has total degree not divisible by $d$.  We will abuse notation by referring to multidegrees as monomials, or monomials of $R$ as multidegrees, wherever it is convenient to do so.
\end{notation}

\begin{remark}
  Since we name our multidegrees by the monomials of $R$ rather than their exponent vectors, the field $\kk$ inside $S_{I}$ is $S_{1}$ (for the unit monomial $1$) rather than $S_{0}$ (for its exponent vector).
\end{remark}

\begin{proposition}\label{p:toricgens}
The toric ideal $S_{I}$ is generated by binomials of the form $Y-Y'$, where $Y$ and $Y'$ have the same multidegree.
\end{proposition}

The Rees ideal of $I$ is defined similarly to the toric ideal, but allows extra bookkeeping.

\begin{definition}\label{d:rees}
  The \emph{Rees algebra} of $I$ is the $R$-subalgebra
  \[
  \mathcal{R}_{I}=R\oplus It \oplus I^{2}t^{2} \oplus \dots \subset R[t].
  \]

  More carefully, let $\mathcal{R}$ denote the polynomial ring $\kk[x_{1},\dots, x_{n}][Y_w : w\in \gens(I)]$, with a new variable $Y_w$ for each generator $w$ of $I$.  The \emph{Rees map} is the $R$-algebra map $\rho: \mathcal{R} \to R[t]$ given by 
\[
\rho(Y_{w})=wt,
\]
and extended algebraically.  (As with the toric ideal, we allow for a free re-indexing of the $w_i$ to avoid unwieldy double subscripts.) The \emph{Rees ideal} of $I$, which we write $\mathcal{R}(I)$, is the kernel of $\rho$.  The \emph{Rees algebra} of $I$, denoted $\mathcal{R}_{I}$, is the quotient $\mathcal{R}_{I}=\mathcal{R}/\mathcal{R}(I)$ (naturally isomorphic to the image of $\rho$, which is a subring of $R[t]$).
\end{definition}

\begin{notation}\label{n:reesmultideg}
The Rees algebra and Rees ideal have two natural gradings.  The \emph{multidegree} is again inherited from $R$ by setting the multidegree of $Y_{w}$ to $w$.  (We also set the multidegree of $x_{i}$ equal to $x_{i}$, and, remembering that we write multidegree multiplicatively, set the multidegree of $t$  to $1$.)  The $t$-degree, which counts the number of $Y$'s in a monomial, is simply the exponent on $t$.
\end{notation}

\begin{remark}\label{r:defeqnrees}
The defining equations of the Rees ideal are considerably more complex than those of the toric ideal.  Essentially, the generators with $t$-degree $d$ correspond to minimal first syzygies on $I^{d}$ (multiplied by $t^{d}$).  Fortunately, despite the increased complexity, a result of Herzog, Hibi, and Vladoiu (which we discuss in Section~\ref{s:reesproof}) allows us to lift our result from the toric ideal to the Rees ideal.
\end{remark}


\section{Fiber graphs for a toric ideal}\label{s:graphs}

Throughout this section, fix a monomial ideal $I\subset R$, its toric ideal $T(I)$, and a multidegree $\mu$.  We define a directed graph that will help us study the Gr\"obner basis of $T(I)$.

\begin{defn}\label{d:fibergraph}
  Define the \emph{fiber graph of $I$ at $\mu$},  $\Gamma_\mu=\Gamma_\mu(I)$ as follows: The vertices of $\Gamma_\mu$ are the monomials of $S_\mu$, where $Z = Y_{w_1}Y_{w_2} \cdots Y_{w_t}$ and $Z' = Y_{w_1'}Y_{w_2'} \cdots Y_{w_t'}$ are connected by an edge whenever $Z$ can be obtained from $Z'$ by performing a Borel move on one of its factors, and the corresponding reverse Borel move on another.  More precisely, $Z$ and $Z'$ are connected by an edge if there exists a reindexing of the monomials such that $w_1$ is obtained from $w_1'$ by a Borel move, $w_2$ is obtained from $w_2'$ via a reverse Borel move, and $w_i = w_i'$ for $i > 2$.  
\end{defn}


\begin{defn}\label{d:fiberdigraph}
Fix in addition a monomial order $\prec$ on $S_{I}$.   We direct each edge of $\Gamma_\mu$ to point to the later of the two monomials.  That is, if $(Z, Z')$ is an edge of $\Gamma_\mu$ and $Z \prec Z'$, we direct the edge towards $Z$.  We write $\vec{\Gamma}_\mu$ for these directed graphs. 
%
\end{defn}

\begin{example}\label{e:fibergraph}
  Consider $I=\Borel(a^{2}c^{3},b^{4}c)$ and $\mu=a^{3}b^{9}c$.  Then $I$ has 14 minimal monomial generators, of which nine figure in the  multidegree $\mu$.  The graph $\vec{\Gamma}_\mu$ is shown in Figure \ref{f:fibergraph}.  The order $\prec$ is the reverse lex order induced from the variable order
  \[
  Y_{b^{4}c}\succ Y_{b^{5}} \succ Y_{ab^{3}c} \succ Y_{ab^{4}} \succ Y_{a^{3}bc} \succ Y_{a^{3}c^{2}} \succ Y_{a^{2}b^{2}c} \succ Y_{a^{2}bc^{2}} \succ Y_{a^{2}c^{3}}.
  \]
The reason for this choice of order is discussed in Section~\ref{s:toricproof}.
\end{example}

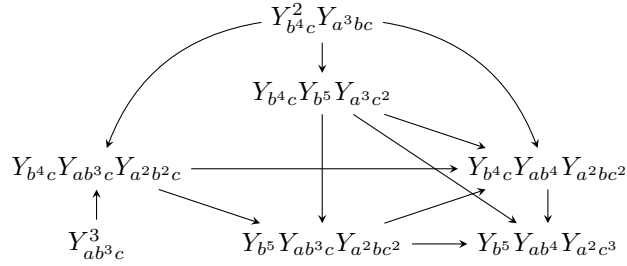
\begin{figure}[hbt]
  \begin{center}
    \begin{tikzpicture}[scale=1,->,>=stealth]
      \node (0) at (0,0) {$Y_{b^{4}c}^2Y_{a^{3}bc}$};
      \node (1) at (0,-1) {$Y_{b^{4}c}Y_{b^{5}}Y_{a^{3}c^{2}}$};
      \node (2) at (-3,-2) {$Y_{b^{4}c}Y_{ab^{3}c}Y_{a^{2}b^{2}c}$};
      \node (3) at (3,-2) {$Y_{b^{4}c}Y_{ab^{4}}Y_{a^{2}bc^{2}}$};
      \node (4) at (-3,-3) {$Y_{ab^{3}c}^3$};
      \node (5) at (0,-3) {$Y_{b^{5}}Y_{ab^{3}c}Y_{a^{2}bc^{2}}$};
      \node (6) at (3,-3) {$Y_{b^{5}}Y_{ab^{4}}Y_{a^{2}c^{3}}$};
      
      \draw (0) to [bend right] (2);
      \draw (2)--(5);
      \draw (0) to [bend left] (3);
      \draw (0)--(1);
      \draw (1)--(3);
      \draw (1)--(5);
      \draw (4)--(2);
      \draw (5)--(6);
      \draw (5)--(3);
      \draw (3)--(6);
      \draw (1)--(6);
      \draw (2)--(3);
    \end{tikzpicture}
    \caption{The fiber graph $\vec{\Gamma}_{a^{3}b^{9}c^{3}}$ for $I=\Borel(a^{2}c^{3}, b^{4}c)$.}    
\label{f:fibergraph}
    \end{center}
\end{figure}

\begin{remark}
  We make two observations here: First, if $I$ is generated in degree $d$, then $\Gamma_\mu$ (and hence $\vec{\Gamma}_\mu$) is empty if the total degree of $\mu$ is not a multiple of $d$.  Second, since the orientation of each $\Gamma_\mu$ is based on a monomial order, it follows that these graphs must be directed acyclic.
\end{remark}

  Our main reason for considering these graphs comes from the following observations of Blasiak \cite{B08} and the fourth author \cite{S11}.

\begin{prop}\label{prop:GQuadratic}
  Suppose every nonempty $\vec{\Gamma}_{\mu}$ has a unique sink.  Then $T(I)$ has a Gr\"obner basis, under $\prec$, consisting of quadric binomials.
\end{prop}

\begin{proof}
  First, we show that the quadric binomials generate $T(I)$.  Clearly, $T(I)$ is generated by binomials of the form $Z-Z'$ with $\phi(Z)=\phi(Z')$.  Now suppose that $Z-Z'$ is such a binomial, and let $\mu=\phi(Z)=\phi(Z')$.  Then $\vec{\Gamma}_\mu$ is nonempty and by assumption contains a unique sink.  It follows that $\Gamma_{\mu}$ is connected, so there exists a path from $Z$ to $Z'$,  $Z = Z_0, Z_1, Z_2, \ldots, Z_k = Z'$, with each $Z_{i}$ adjacent to $Z_{i+1}$.

  By construction of $\Gamma_{\mu}$, we may write $Z_{i}=Y_{w_{i,1}}\dots Y_{w_{i,t}}$ and $Z_{i+1}=Y_{w_{i+1,1}}\dots Y_{w_{i+1,t}}$, where $w_{i,1}=\B{s}{r}(w_{i+1,1})$, $w_{i,2}=\RB{r}{s}(w_{i+1,2})$, and $w_{i,j}=w_{i+1,j}$ for all other $j$.  Consequently, $Z_{i}-Z_{i+1}$ is contained in the ideal generated by the quadric binomial $Y_{w_{i,1}}Y_{w_{i,2}} - Y_{w_{i+1,1}}Y_{w_{i+1,2}}$.  In particular, $Z-Z' = \sum\left(Z_{i}-Z_{i+1}\right)$ is contained in the ideal generated by all quadric binomials of $T(I)$.


  Now we verify that the set of quadric binomials satisfies Buchberger's criterion.  Fix two such quadric binomials $Q$ and $Q'$.  Then their $S$-polynomial is the multihomogeneous binomial $Z-Z'$ in some multidegree $\mu$.  It follows that $\vec{\Gamma}_{\mu}$ is nonempty, so by assumption it has a sink $Z^{\ast}$.  Then there is a path from $Z$ to $Z^{\ast}$, $Z = Z_0, Z_1, Z_2, \ldots, Z_k = Z^{\ast}$, with each $Z_{i}$ adjacent to $Z_{i+1}$ and $Z_{i}\succ Z_{i+1}$.

  Following the argument above, write \[Z_{i}-Z_{i+1} = (Y_{w_{i,1}}Y_{w_{i,2}} - Y_{w_{i+1,1}}Y_{w_{i+1,2}})Y_{w_{i,3}}\dots Y_{w_{i,t}},\] and observe that the leading term of the quadric binomial $Y_{w_{i,1}}Y_{w_{i,2}} - Y_{w_{i+1,1}}Y_{w_{i+1,2}}$ is $Y_{w_{i,1}}Y_{w_{i,2}}$, which divides $Z_{i}$.  Thus, any polynomial containing $Z_{i}$ may be reduced by the quadric binomial $Y_{w_{i,1}}Y_{w_{i,2}} - Y_{w_{i+1,1}}Y_{w_{i+1,2}}$, and the result replaces $Z_{i}$ with $Z_{i+1}$.  Inductively, our $S$-polynomial $Z-Z'$ reduces, modulo the quadric binomials, to $Z'-Z^{\ast}$.  By a similar argument, it then reduces to $Z^{\ast}-Z^{\ast}=0$.  We conclude that the quadric binomials satisfy Buchberger's criterion and so form a Gr\"obner basis for $T(I)$.
\end{proof}

\begin{remark}\label{r:NoetherianReduction}
Proposition~\ref{prop:GQuadratic} may also be proved using \textit{coherent marking} and \textit{Noetherian reduction relations} (see~\cite[Chapter~3]{S96}).
\end{remark}

\section{Two-Borel Ideals}\label{s:toricproof}

In this section, we construct a monomial order on $S_{I}$ under which all the (nonempty) directed fiber graphs $\vec{\Gamma}_{\mu}$ have a unique sink.  It then follows from Proposition \ref{prop:GQuadratic} that the toric ideal $T(I)$ has a quadric Gr\"obner basis with respect to this order and in particular is Koszul.  

We begin with a technical lemma, which will allow us to to assume that the Borel generators of $I$ always divide the multidegree $\mu$ when we study the graph $\vec{\Gamma}_{\mu}$.

\begin{lem}\label{lem:borel2}
Let $\mu$ be a multi-degree, $M$ a monomial of degree $d$, $M'$ the lex-latest degree $d$ monomial in $\borel(M)$ such that $M' \text{ divides }\mu$, and $m$ another monomial in $\borel(M)$ so that $m \text{ divides } \mu$.  Then $m\in\borel(M')$.  Furthermore, if $m\neq M'$ and $j$ is the largest index so that $\sigma_j(m)<\sigma_j(M')$, then there is some index $i<j$ so that $\RB{i}{j}(m)\in\borel(M')$ and $\RB{i}{j}(m) \text{ divides }\mu$.
\end{lem}
\begin{proof}
  For the first statement, suppose to the contrary that $m\notin \borel(M')$.  Let $j$ be the largest index so that $\sigma_j(m)>\sigma_j(M')$ (by Lemma~\ref{lem:borel1} there is at least one such index).  We have $\sigma_j(M)\ge\sigma_j(m)>\sigma_j(M')$, so again by Lemma~\ref{lem:borel1}, there is some index $i<j$ so that $M''=\frac{x_j}{x_i}M'\in\borel(M)$.

  We claim $M''\text{ divides } \mu$.  Indeed, write $M'=x_1^{a_1}\cdots x_n^{a_n}$; we need only show that $x_{j}^{a_{j}+1}$ divides $\mu$.  To this end, write $m=x_1^{b_1}\cdots x_n^{b_n}$; since $m$ divides $\mu$, it is enough to show that $b_{j}> a_{j}$.  But by the construction of $j$ we have $b_{j} + \sigma_{j+1}(m) = \sigma_{j}(m) > \sigma_{j}(M') = a_{j} + \sigma_{j+1}(M')$ (taking $\sigma_{j+1}=0$ if $j=n$).  Also by the choice of $j$, we have $\sigma_{j+1}(m)\leq \sigma_{j+1}(M')$.  It follows that $b_{j}>a_{j}$ as desired.

  Now we have that $M''=\RB{i}{j}(M')$, which is lex-later than $M'$, is contained in $\Borel(M)$ and divides $\mu$.  But this contradicts the choice of $M'$ as the lex-last monomial in $\Borel(M)$ dividing $\mu$.  Thus $m\in\borel(M')$ as desired.


For the second claim, let $j$ be the largest index with $\sigma_{j}(m)>\sigma_{j}(M')$.  Then by Lemma \ref{lem:borel1}, there is some $i$ with $\RB{j}{i}(m)\in \Borel(M')$.  By the same reasoning as above, we conclude that $\RB{j}{i}(m)$ divides $\mu$.

\end{proof}


\begin{cor}\label{cor:GraphReduction}
Suppose $I=\Borel(M_1,M_2,\ldots,M_k)$ and $\mu$ is a multi-degree.  Let $M'_1,\ldots,M'_k$ be the lex-last monomials of $\borel(M_1),\ldots,\borel(M_k)$, respectively, that divide the multidegree $\mu$, and set $I'=\borel(M'_1,\ldots,M'_k)$.  Then $\vec{\Gamma}_\mu(I)=\vec{\Gamma}_\mu(I')$.
\end{cor}

For the duration of the paper, let $I=\Borel(M,N)$ be a two-Borel ideal equigenerated in degree $d$, and recall that the toric ideal $T(I)$ is an ideal of the ring $S_{I}=\kk[Y_w: w\in\mbox{gens}(I)]$.  The structure of the directed graphs $\vec{\Gamma}_{\mu}$ depends on the choice of term order for the ring $S_{I}$, so our first order of business is to define one that allows our arguments to work.  


\begin{definition}\label{d:fibersinkorder}\label{conv:MonomialOrderConvention}
  We define the \emph{fiber sink order} for $S_{I}$ as follows.  First, partition the minimal monomial generators of $I$ into two sets, $G_{M}=\gens(\Borel(M))$ and $G_{N}=\gens(I)\smallsetminus G_{M}$. (Note that $G_M$ is closed under going down in the Borel order, while $G_N$ is not; see, e.g., Figure~\ref{fig:MonomialOrderConvention}.) Then order the variables $Y_{w}$ of $S_{I}$ according to the \emph{fiber sink variable order}, defined as follows.
\begin{itemize}
\item If $u,v\in G_N$ and $u$ precedes $v$ in the lex order, then $Y_{v}$ precedes $Y_{u}$.
\item If $u\in G_N, v\in G_M$, then $Y_v$ precedes $Y_u$
\item If $u,v\in I_M$ and $u$ precedes $v$ in the lex order, then $Y_{u}$ precedes $Y_{v}$.
\end{itemize}
  Finally, the fiber sink order on $S_{I}$ is the reverse lex order induced by the fiber sink variable order.
\end{definition}

\begin{example}
  Let $M=a^{2}c^{3}$, $N=b^{4}c$, and $I=\Borel(M,N)$.  The graph on the left in Figure \ref{fig:MonomialOrderConvention} is the Hasse diagram of the Borel order on the minimal monomial generators of $I$.  The elements of $G_{N}$ are in the red circles, and the elements of $G_{M}$ are in green boxes.  The labels $Y_{0}, \dots, Y_{13}$ are the fiber sink variable order on these generators, with $Y_{0}=N$ first and $Y_{13}=M$ last.  Finally, the graph on the right is $\vec{\Gamma}_{\mu}$ for the multidegree $\mu=a^{3}b^{9}c^{3}$, using the fiber sink order.
\end{example}

\begin{figure}[hbt]
\captionsetup[subfigure]{labelformat=empty}
\centering

\begin{subfigure}[c]{.49\textwidth}
\begin{tikzpicture}
\tikzstyle{every node}=[]
\node[fill=red!50, circle, inner sep=0 pt, minimum size=20 pt] (0) at (3,0) {$b^4c$};
\node[fill=red!50, circle, inner sep=0 pt, minimum size=20 pt] (1) at (4,-1) {$b^5$};
\node[fill=red!50, circle, inner sep=0 pt, minimum size=20 pt] (2) at (2,-1) {$ab^3c$};
\node[fill=red!50, circle, inner sep=0 pt, minimum size=20 pt] (3) at (3,-2) {$ab^4$};
\node[fill=green,rectangle,inner sep=3 pt] (4) at (1,-6) {$a^5$};
\node[fill=green,rectangle,inner sep=3 pt] (5) at (1,-5) {$a^4b$};
\node[fill=green,rectangle,inner sep=3 pt] (6) at (0,-4) {$a^4c$};
\node[fill=green,rectangle,inner sep=3 pt] (7) at (2,-4) {$a^3b^2$};
\node[fill=green,rectangle,inner sep=3 pt] (8) at (0,-3) {$a^3bc$};
\node[fill=green,rectangle,inner sep=3 pt] (9) at (-1,-2) {$a^3c^2$};
\node[fill=green,rectangle,inner sep=3 pt] (10) at (2,-3) {$a^2b^3$};
\node[fill=green,rectangle,inner sep=3 pt] (11) at (1,-2) {$a^2b^2c$};
\node[fill=green,rectangle,inner sep=3 pt] (12) at (0,-1) {$a^2bc^2$};
\node[fill=green,rectangle,inner sep=3 pt] (13) at (0,0) {$a^2c^3$};
\node[right] at (0.east) {$Y_0$};
\node[right] at (1.east) {$Y_1$};
\node[right] at (2.east) {$Y_2$};
\node[right] at (3.east) {$Y_3$};
\node[right] at (4.east) {$Y_4$};
\node[right] at (5.east) {$Y_5$};
\node[right] at (6.east) {$Y_6$};
\node[right] at (7.east) {$Y_7$};
\node[right] at (8.east) {$Y_8$};
\node[right] at (9.east) {$Y_9$};
\node[right] at (10.east) {$Y_{10}$};
\node[right] at (11.east) {$Y_{11}$};
\node[right] at (12.east) {$Y_{12}$};
\node[right] at (13.east) {$Y_{13}$};

\node at (0,.5) {$M$};
\node at (3,.5) {$N$};

\draw (0)--(1)--(3)--(2)--(0);
\draw (3)--(10)--(11)--(2);
\draw (10)--(7)--(5)--(4);
\draw (13)--(12)--(9)--(8)--(6)--(5);
\draw (12)--(11)--(8);
\end{tikzpicture}
\caption{Hasse diagram for generators of $I$}
\end{subfigure}
\begin{subfigure}[c]{.49\textwidth}
%

\begin{tikzpicture}[scale=1,->,>=stealth]
\node (0) at (0,0) {$Y_0^2Y_8$};
\node (1) at (0,-1) {$Y_0Y_1Y_9$};
\node (2) at (-2,-2) {$Y_0Y_2Y_{11}$};
\node (3) at (2,-2) {$Y_0Y_3Y_{12}$};
\node (4) at (-2,-3) {$Y_2^3$};
\node (5) at (0,-3) {$Y_1Y_2Y_{12}$};
\node (6) at (2,-3) {$Y_1Y_3Y_{13}$};

\draw (0) to [bend right] (2);
\draw (2)--(5);
\draw (0) to [bend left] (3);
\draw (0)--(1);
\draw (1)--(3);
\draw (1)--(5);
\draw (4)--(2);
\draw (5)--(6);
\draw (5)--(3);
\draw (3)--(6);
\draw (1)--(6);
\draw (2)--(3);
\end{tikzpicture}

\caption{The graph $\vec{\Gamma}_\mu$ for $\mu=a^3b^9c^3$}
\end{subfigure}

\caption{The fiber sink order on $S_{I}$, with $I=\Borel(a^{2}c^{3},b^{4}c)$.}\label{fig:MonomialOrderConvention}
\end{figure}

\begin{remark}
  The fiber sink variable order always begins with $Y_{0}=N$, continues with the other elements of $G_{N}$ in antilex order, then takes the elements of $G_{M}$ in lex order, ending with $M$.  Heuristically, paths from $N$ to $M$ in the Hasse diagram of the Borel order represent chains in the fiber sink variable order.

  In fact, our arguments do not rely on the use of the lex order within $G_{N}$ and $G_{M}$.  Any linear extension of the antiborel order on $G_{N}$ followed by any linear extension of the Borel order on $G_{M}$ will yield the same unique sinks in every $\vec{\Gamma}_{\mu}$.
\end{remark}

\begin{remark}
  While we think of $M$ and $N$ as being incomparable in the Borel order, this assumption is not actually necessary.  In the degenerate case where $M$ and $N$ are comparable, then $I=\Borel(M)$ or $I=\Borel(N)$ is a principal Borel ideal, and consequently many orders on the toric ideal of a principal Borel ideal can occur as the fiber sink order.  For example, setting $M=a^{d}$ yields the antilex order on the generators of $I$, and setting $N=a^{d}$ yields the lex order.
\end{remark}

\begin{remark}
We could have dualized everything in the creation of the fiber sink order.  If we started with the antilex order on $G_{M}$ and ended with the lex order on $G_{N}$, and then induced the lex instead of the reverse lex order on the $Y$ variables, our graphs would end up having unique sources instead of sinks.  But apart from the change in language all our subsequent arguments would go through without modification.
\end{remark}

Now we study the sinks in the graphs $\vec{\Gamma}_{\mu}$.  Observe from Figure \ref{fig:MonomialOrderConvention} that movement along the directed edges always consists of replacing a low-indexed variable with a lower index, and replacing a high-indexed variable with a higher index.  Thus there are two possible heuristics, depending on which of these we view as a goal, and which we view as incidental.


\begin{defn}\label{defn:ReplacementTypes}
  Write $Y=
  Y_{w_{1}}\dots Y_{w_{k}}$,
  with $Y_{w_{1}}$ first and $Y_{w_{k}}$ last in the fiber sink variable order.  
  We say $Y$ has \textit{type $N$} if every $w_i\in G_N$.  We say $Y$ has \textit{type $M$} if $w_k\in G_M$.    (So every $Y$ has type $M$ or $N$, but not both.)  If $Y$ has type $N$ then a \textit{type $N$ replacement} of $Y$ is the monomial resulting from applying a reverse Borel move to $Y_{w_{1}}$ and the corresponding Borel move to one of the other variables.  (That is, we replace $Y_{w_{1}}$ with $Y_{\RB{i}{j}(w_{1})}$ and replace some other $Y_{w_{\ell}}$ with $Y_{\B{j}{i}(w_{\ell})}$, while leaving any other factors untouched.)
  Similarly, if  $Y$ has type $M$ then a \textit{type $M$ replacement} of $Y$ is the monomial resulting from applying a reverse Borel move to $Y_{w_{k}}$ and the corresponding Borel move to one of the other variables.
\end{defn}


The following lemma motivates the choice of reverse lex order on the $Y$ variables in the fiber sink order.

\begin{lem}\label{lem:whygrevlex}
  Fix a multidegree $\mu$ divisible by both $M$ and $N$, and suppose $Y\in S_{\mu}$.  If $Y$ has type $N$ and $Y_{N}$ does not divide $Y$, then a type $N$ replacement is possible.  If $Y$ has type $M$ and $Y_{M}$ does not divide $Y$, then a type $M$ replacement is possible.  In either case, the replacement is later than $Y$ in the fiber sink order.
\end{lem}
\begin{proof}
  We suppose $Y$ has type $M$ but is not divisible by $Y_{M}$, and prove the conclusions of the lemma.  (The argument for type $N$ is identical.)

  Write $Y=Y_{w_{1}}\dots Y_{w_{k}}$ with $Y_{w_{k}}$ last in the fiber sink variable order.  By assumption, $w_{k}\neq M$, but $w_{k}\in\Borel(M)$, so by Lemma \ref{lem:borel2} there exists a reverse Borel move $\RB{i}{j}$ such that $=\RB{i}{j}(w_{k})$ is contained in $\Borel(M)$ and divides $\mu$.  Write $w_{k}=\prod x_{\ell}^{a_{\ell}}$; it follows that $x_{j}^{a_{j}+1}$ divides $\mu$ and in particular $x_{j}$ divides one of the other $w_{r}$.  Now we can perform the Borel move $\B{j}{i}$ on $w_{r}$, and the result will be contained in $I$.  So the monomial
  \[Y'=Y_{\RB{i}{j}(w_{k})} Y_{\B{j}{i}(w_{r})} \prod_{s\neq r,k}Y_{w_{s}}\]
is a type $M$ replacement of $Y$ and is in the fiber $S_{\mu}$.  In particular, this replacement is possible.  Finally, observe that $Y'$ is divisible by $Y_{\RB{i}{j}(w_{k})}$, which comes after $Y_{w_{k}}$ in the fiber sink variable order.  Since $Y_{w_{k}}$ is the last variable dividing $Y$ and the fiber sink order is reverse lex, we conclude that $Y'$ comes after $Y$ in the fiber sink order.
\end{proof}



\begin{corollary}\label{c:describesinks}
Suppose $I=\Borel(M,N)$ and fix a multidegree $\mu$ divisible by both $M$ and $N$.  Then every sink of type $M$ in $\vec{\Gamma}_{\mu}$ is divisible by $Y_{M}$, and every sink of type $N$ is divisible by $Y_{N}$.  Furthermore, every sink in $\vec{\Gamma}_{\mu}$ has the form $Y_{M}Z_{M}$, where $Z_{M}$ is a sink in $\vec{\Gamma}_{\frac{\mu}{M}}$, or $Y_{N}Z_{N}$, where $Z_{N}$ is a sink in $\vec{\Gamma}_{\frac{\mu}{N}}$.
\end{corollary}

\begin{proof}
  The first claim is immediate from Lemma \ref{lem:whygrevlex}.  For the second claim, observe that the subgraph of $\vec{\Gamma}_{\mu}$ containing $Y_{M}$ is isomorphic to $\vec{\Gamma}_{\frac{\mu}{M}}$.
\end{proof}

Corollary \ref{c:describesinks} suggests an algorithm to find the sinks of $\vec{\Gamma}_{\mu}$:  If there are monomials of type $M$, find the sinks of $\vec{\Gamma}_{\frac{\mu}{M}}$ (after replacing $M$ and $N$ by $M'$ and $N'$ as in Corollary \ref{cor:GraphReduction} if necessary),
and multiply by $Y_{M}$.  If there are monomials of type $N$, multiply the sinks of $\vec{\Gamma}_{\frac{\mu}{N}}$ by $Y_{N}$.

To prove that every $\vec{\Gamma}_{\mu}$ has a unique sink, it suffices to remove the choice between types $M$ and $N$.

\begin{proposition}\label{p:uniquesink}\label{thm:AllN}
  Suppose $I=\Borel(M,N)$ and fix a multidegree $\mu$ divisible by both $M$ and $N$.  Then either every sink in $\vec{\Gamma}_{\mu}$ has type $M$, or every sink has type $N$.  
\end{proposition}

\begin{proof}
  Suppose that $\vec{\Gamma}_{\mu}$ has a sink of type $N$.  We will show that in fact it has no monomials of type $M$.

  To that end, suppose that $Y=Y_{n_{1}}\dots Y_{n_{t}}$ is a sink, so $\mu=n_{1}\dots n_{t}$ with each $n_{i}\in G_{N}$.  Without loss of generality, we may assume that $n_{t}$ is the last of these factors in the fiber sink order, i.e., first in the lex order.  Since
%
  $n_{t} \notin \borel(M)$, there must be an index $j$ with 
\[
\sigma_j(n_{t}) > \sigma_j(M).
\]
Let $j$ be the maximal such index.   Since $n_{t} \in \borel(N)$, we have $\sigma_j(M) < \sigma_j(n_{t}) \leq \sigma_j(N)$. 
Moreover, observe  that $\sigma_j(n_{t}) > \sigma_{j+1}(n_{t})$, since otherwise we would have $\sigma_{j+1}(n_{t}) = \sigma_j(n_{t}) > \sigma_j(M) \geq \sigma_{j+1}(M)$, contradicting the maximality of $j$.  Thus $x_j \text{ divides } n_{t}$.  
	
The key ingredient in our proof is that, for any $s\neq t$, we have
\[
\sigma_j(n_{s}) = \sigma_j(N).  
\]
	
Indeed, suppose otherwise.  Then, as $n_{s} \in \borel(N)$, we would have $\sigma_j(n_{s}) < \sigma_j(N)$.  By Lemma~\ref{lem:borel1} there would then exist an $i< j$ such that $\frac{x_j}{x_i}n_{s} \in \borel(N)$.  Meanwhile, since $x_{j}$ divides $n_{t}$, we also have $\B{j}{i}(n_{t})=\frac{x_i}{x_j}n_{t} \in \borel(N)$.  But then
replacing $n_{s}$ with $\RB{i}{j}(n_{s})$ and $n_{t}$ with $w=\B{j}{i}(n_{t})$ creates a new element $Y'\in S_{\mu}$.  Now $Y_{w}$ comes after $Y_{n_{t}}$ in the fiber sink variable order, and $Y_{n_{t}}$ is the last variable dividing $Y$, so $Y'$ comes after $Y$ in the fiber sink order.  Thus, the directed edge from $Y$ to $Y'$ contradicts the assumption that $Y$ is a sink.



We now prove that $S_{\mu}$ has no elements of type $M$.  Indeed, suppose we can write 
$\mu = m \cdot \prod_{k = 1}^{t-1} w_k$
with $m \in G_{M}$ and each $w_k \in \borel(I)$.  We have
\[
\sigma_j(\mu) = \sigma_j(m) + \sum_{k=1}^{t-1}\sigma_j(w_k) = \sum_{s = 1}^t \sigma_j(n_s) = \sigma_{j}(n_{t}) + \sum_{k=1}^{t-1}\sigma_{j}(n_{s}). 
\]
As $m \in \borel(M)$, we must have $\sigma_j(m) \leq \sigma_j(M) < \sigma_j(n_t)$.  
Consequently,
\[
\sum_{k=1}^{t-1}\sigma_j(w_k) > \sum_{s = 1}^{t-1} \sigma_j(n_s) = (t-1)\sigma_j(N). 
\]
By the pigeonhole principle, this means $\sigma_{j}(w_{k})>\sigma_{j}(N)$ for some $k$.  Immediately, we have $w_{k}\not\in\Borel(N)$, so $w_{k}\in\Borel(M)$.  But then we have $\sigma_{j}(w_{k})\leq \sigma_{j}(M)<\sigma_{j}(n_{t})\leq \sigma_{j}(N)$, a contradiction.

Thus, the existence of a sink of type $N$ prevents the existence of any elements of type $M$, as desired.
\end{proof}

Putting all this together, we conclude that the toric ideal of a two-Borel ideal is Koszul.

\begin{thm}\label{thm:mainKoszul}
Let $I=\Borel(M,N)$ be a two-Borel ideal.  Then $T(I)$ has a quadratic Gr\"obner basis in the fiber sink order.  In particular, the toric ring $\kk[I]$ is Koszul.
\end{thm}

\begin{proof}
  By Proposition 4.5, it suffices to show that every nonempty $\vec{\Gamma}_{\mu}$ has a unique sink.  If $\mu=1$, then $\vec{\Gamma}_{\mu}$ consists of a single point and there is nothing to prove.  Otherwise, we induct on divisibility.

  By Corollary 5.2, we may if necessary replace $M$ and $N$ with $M'$ and $N'$, the lex-last elements of $\Borel(M)$ and $\Borel(N)$ dividing $\mu$.

  By Proposition 5.11, either every sink has type $M$ or every sink has type $N$.  In the first case, every sink has the form $Y_{M}Z_{M}$, where $Z_{M}$ is a sink of $\vec{\Gamma}_{\frac{\mu}{M}}$.  In the second case, every sink has the form $Y_{N}Z_{N}$, where $Z_{N}$ is a sink of $\vec{\Gamma}_{\frac{\mu}{N}}$.  Inductively, both $\vec{\Gamma}_{\frac{\mu}{M}}$ and $\vec{\Gamma}_{\frac{\mu}{N}}$ have unique sinks.  Consequently, in either case, $\vec{\Gamma}_{\mu}$ has a unique sink.
\end{proof}

\section{Concluding Remarks}\label{s:reesproof} 

We now make the connection between Koszulness of the toric ring $\kk[I]$ and Koszulness of the Rees algebra $\mathcal{R}_I$.  The following result, due to Herzog, Hibi, and Vladoiu, allows a straightforward relationship.

\begin{thm}\cite[Theorem~5.1]{HHV05}\label{thm:ReductionToToric}
Let $I$ be a Borel  ideal minimally generated by $\{w_1,\ldots,w_k\}$, with Rees ideal $\mathcal{R}(I)$ as in Definition \ref{d:rees}.  Let $<$ be any term order on $\kk[Y_{w_1},\ldots,Y_{w_k}]$ and let $<_{\mathcal{R}}$ be the elimination order on $\mathcal{R}$ defined by $x^\alpha Y^\beta<x^\gamma Y^\delta$ whenever either $x^\alpha<_{\lex}x^\gamma$ or $x^\alpha=x^\gamma$ and $Y^\beta<Y^\delta$ in $\kk[Y_{w_1},\ldots,Y_{w_k}]$.  Suppose $\mathcal{G}_{<}(T(I))$ is a Gr\"obner basis for the toric ideal $T(I)$ in $\kk[Y_{w_1},\ldots,Y_{w_k}]$ with respect to $<$.  Then
\[
\{x_jY_{w_t}-x_iY_{w_u} : x_{j}w_{t}=x_{i}w_{u}  \}\cup \mathcal{G}_{<}(T(I))
\]
is a Gr\"obner basis for the Rees ideal $\mathcal{R}(I)$ with respect to $<_S$ in $S$.
\end{thm}

Theorem \ref{thm:ReductionToToric} completes the proof that the Rees ideal of a two-Borel ideal is Koszul.

\begin{cor}\label{t:thegoal}
The Rees algebra of a two-Borel ideal is Koszul.
\end{cor}

\begin{proof}
  Let $I$ be a two-Borel ideal.  By Theorem \ref{thm:mainKoszul}, the toric ideal $T(I)$ has a quadratic Gr\"obner basis.  By Theorem \ref{thm:ReductionToToric}, the Rees ideal of $I$ has a quadratic Grobner basis, 
  and in particular is Koszul.
\end{proof}

We close with some questions for future research.

\begin{question}
If $I$ is a two-Borel ideal, is the toric ring $\kk[I]$ Cohen-Macaulay?
\end{question}

\begin{question}
Let $I_1,\ldots,I_k$ be ideals in a polynomial ring, where each $I_j$ is either a principal Borel ideal or a two-Borel ideal.  Is the multi-Rees algebra $\mathcal{R}_{I_1,\ldots,I_k}$ (see~\cite{BC16}) Koszul?  What about Cohen-Macaulay or normal?
\end{question}

\noindent \textbf{Acknowledgments}: We thank Aldo Conca for introducing us to this problem in his talk at the conference honoring Craig Huneke in July 2016, and we thank Travis Grigsby for helpful conversations related to his master's thesis. We performed the computations leading to this paper in Macaulay2 \cite{M2}, often using resources from the Oklahoma State University High Performance Computing Center, which is supported in part through National Science Foundation grant ACI--1126330. The work in this paper was partially supported by grants from the Simons Foundation (\#422465 to Christopher Francisco and \#202115 to Jeffrey Mermin) and an Oklahoma State University College of Arts \& Sciences Summer Research Program grant (to Jay Schweig). 

\bibliography{ReesBib}{}
\bibliographystyle{plain}

\end{document}